\documentclass[fleqn,final]{zzz}
\usepackage{mathrsfs}
\usepackage{color}
\usepackage{tikz}
\usepackage{graphicx}
\usepackage{rotating}
\usepackage[pdftex]{hyperref}

\hypersetup{
 colorlinks = true,
 urlcolor = blue!60!black,
 linkcolor = red!60!black,
 citecolor = green!60!black,
}

\newtheorem{thm}{Theorem}[section]

\newtheorem{lem}[thm]{Lemma}
\newtheorem{rem}[thm]{Remark}
\newtheorem{defn}[thm]{Definition}

\makeatletter
\def\eqnarray{\stepcounter{equation}\let\@currentlabel=\theequation
\global\@eqnswtrue
\tabskip\@centering\let\\=\@eqncr
$$\halign to \displaywidth\bgroup\hfil\global\@eqcnt\z@
 $\displaystyle\tabskip\z@{##}$&\global\@eqcnt\@ne
 \hfil$\displaystyle{{}##{}}$\hfil
 &\global\@eqcnt\tw@ $\displaystyle{##}$\hfil
 \tabskip\@centering&\llap{##}\tabskip\z@\cr}

\def\endeqnarray{\@@eqncr\egroup
 \global\advance\c@equation\m@ne$$\global\@ignoretrue}

\def\@yeqncr{\@ifnextchar [{\@xeqncr}{\@xeqncr[5pt]}}
\makeatother

\begin{document}

\renewcommand{\PaperNumber}{***}

\FirstPageHeading

\ShortArticleName{Localization of zeros of polar polynomials on the unit disc
}

\ArticleName{Localization of zeros of polar polynomials  on the unit disc
}

\Author{R. S.~Costas-Santos\,$^\dag\!\!\ $
and A. Rehouma\,$^\ddag{}$}
\AuthorNameForHeading{R. S. Costas-Santos  and A. Rehouma
}
\Address{$^\dag$ Departamento de M\'etodos Cuantitativos, 
Universidad Loyola Andaluc\'ia, 
E-41704, Dos Hermanas, Seville, Spain.
} 
\URLaddressD{\href{http://www.rscosan.com}
{http://www.rscosan.com}}
\EmailD{rscosa@gmail.com} 

\Address{$^\ddag$ Department of Mathematics, Faculty of exact sciences, 
University of Hama Lakhdar of Eloued, Algeria.
} 
\URLaddressD{\href{https://sites.google.com/view/mathsrehoumablog/accueil}{https://sites.google.com/view/mathsrehoumablog/accueil}}
\EmailD{rehouma-abdelhamid@univ-eloued.dz} 

\ArticleDates{Received ?? \today \, in final form ????; Published online ????}

\Abstract{
We derive a useful result about the zeros of the $k$-polar 
polynomials on the unit circle; in particular we obtain a ring shaped 
region containing all the zeros of these polynomials. 
Some examples are presented. 
}

\Classification{05C38; 15A15; 05A15; 15A18}

\section{Introduction}
Finding the roots of polynomials is a problem of interest in both mathematics
and in areas of application such as physical systems, which can be reduced 
to solving certain equations. 
There are very interesting geometric relationships between the roots of a 
polynomial $f_n (z) $ and those of $f_n' (z)$. The most important result 
is the following.
\begin{thm}[The Gau{\ss}-Lucas theorem \cite{ZBMATH2717048}] \label{thm:1.1}
Let $f_n(z)\in \mathbb C[z]$ be a polynomial of degree at least one. 
All zeros of  $f_n'(z)$ lie in the convex hull of the zeros of  the zeros of $f_n(z)$.
\end{thm}
The location of zeros, or critical points, of polynomials has many physical 
and geometrical interpretations. For example, C. F. Gau{\ss} in 1816 showed 
that the roots of $f_n' (z) $ are the positions of equilibrium in the field of 
force due to equal particles situated at each root of $f_n(z)$, if each 
particle repels with a force equal to the inverse distance being the inverse 
distance law.

Many results exist concerning the location of the zeros of a polynomial of a 
complex variable as a function of the coefficients of the polynomial. 
One is the well-known Enstrom-Kakeya theorem \cite{MR2105088}. 
Another one, useful to obtain more precise information about the zeros 
of a polynomial, was obtained by J. H. Grace.
\begin{thm}[The Grace's theorem \cite{ZBMATH2659180}]\label{thm:1.2}
Let $a(z)$ and $b(z)$ be the polynomials  
\[
a(z)=\sum_{\ell=0}^n a_\ell {n \choose \ell} z^\ell,\quad 
b(z)=\sum_{\ell=0}^n b_\ell {n \choose \ell} z^\ell. 
\]
If the zeros of both polynomials lie in the unit disk, then the zeros of the 
``composition'' of the two
\[
c(z)=\sum_{\ell=0}^n a_\ell b_\ell {n \choose \ell} z^\ell, 
\]
also lie in the unit disk.
\end{thm}
The zeros of orthogonal polynomials has been an intensively studied 
subject since the beginning of the twentieth century, and several 
breakthroughs have been made in the recent years. 
Barry Simon  \cite{MR2105088} proved 
that if  $\mu$ is a finite positive measure defined on the Borelian 
$\sigma$-algebra of $\mathbb C$, $\mu$ is absolutely continuous with 
respect to the Lebesgue measure $d\theta/(2\pi)$ on $[-\pi,\pi]$; and 
$(L_n(z))$ is the system of monic orthogonal polynomials with 
respect to $\mu$, i.e.,
\begin{eqnarray} 
\label{eq:1}
\int_{-\pi}^{\pi} L_n(z) z^{-j}\, d\mu(\theta)&=&0,\quad j=0, 1, ..., n-1,\\
\dfrac{1}{2\pi} \int_{-\pi}^{\pi} L_n(z)\overline{L_n(z)}\, d\mu(\theta)&=&
\|L_n\|^2\ne 0,\quad n=0, 1, ...,\nonumber
\end{eqnarray}
where $z=\exp(i\theta)$ and $d\mu(\theta)=\rho(\theta)d\theta$ for  
$\rho\in L^1([-\pi,\pi],d\theta)$  a measure supported on the 
unit circle 
$\mathbb T=\{z\in \mathbb C: |z|=1\}$. 

Then, all the zeros of $L_n(z)$ are contained in the 
unit closed disk $\overline{D(0,1)}=\{z : |z| \le 1\}=:\mathbb D$.

We are going to introduce monic $k$-polar polynomials.
\begin{defn} \label{def:1.3} Let $\mu$ be a finite measure defined on the Borelian 
$\sigma$-algebra of $\mathbb C$ such that it contains an infinite 
number of points and let $(L_n(z))$ the system of monic orthogonal 
polynomials with respect to $\mu$.
Let $\xi$ be a fixed complex number. Let $k$ be a positive integer.
The $k$-polar polynomial related to $\mu$, which will be denoted 
by $Q_{n;k}(z;\xi)$,  is the polynomial solution of degree $n$ of 
the $k$-th order linear differential equation 
\[
\dfrac{d^k}{d z^k} (z-\xi)^k P(z)=(n+1)\cdots (n+k) L_n(z).
\]
\end{defn}
\begin{rem}
By construction,  $Q_{n;0}(z)=L_n(z)$ for all $n$.
\end{rem}
In the last years some attention has been paid to the so-called polar orthogonal 
polynomials. Fandora and Pijeira \cite{MR1882659} have studied $1$-polar orthogonal 
polynomial sequences associated with a measure supported on the segment. 
A similar study has been done by Pijeira and Urbina \cite{MR2764236}, in the 
case of $1$-polar Legendre polynomials. 

Our main purpose is to study the location of the zeros of $k$-polar orthogonal 
polynomials on the unit circle, in short OPUC, with respect to a generic measure $\mu$. 

In Section 2
we present some preliminaries and basic results we need to obtain the main result.
In Section 3 we state the main result of this work as well we study the location 
of zeros of three interesting examples of $k$-polar orthogonal polynomials 
on the unit circle. 
Since extensive calculations indicate that these polynomials often have 
complex zeros  and there exist a ring shaped region containing all the 
zeros of polar orthogonal polynoials we present in Section 4 numerical 
calculations to see if Sendov's conjecture \cite[p. 267]{MR700266} holds 
true or  not for such examples. 
\section{Preliminaries}
Given a complex number $z_0\in \mathbb C$ and a radius $r>0$, 
we define the open disk
\[
D(z_0,r):=\{z\in \mathbb C: |z-z_0|<r\},
\]
the closed disk 
\[
\overline{D(z_0,r)}:=\{z\in \mathbb C: |z-z_0|\le r\},
\]
and the circle 
\[
\partial D(z_0,r):=\{z\in \mathbb C: |z-z_0|=r\}.
\]
Let $\mu$ be a finite positive Borel measure that is 
absolutely continuous with respect to the Lebesgue 
measure $d\theta/(2\pi)$ on $[-\pi,\pi]$.
Let $(L_n(z))$ be the system of monic polynomials 
orthogonal with respect to $\mu$. It is known that 
the following mutually 
equivalent recursions relations hold for these polynomials:
\begin{eqnarray*}
L_n(z)&=&zL_{n-1}(z)+L_n(0)L_{n-1}^*(z),\\
L^*_n(z)&=&L^*_{n-1}(z)+z\overline{L_n(0)}L_{n-1}(z),
\end{eqnarray*}
where 
\begin{equation} \label{eq:2}
L_n^*(z)=z^n \overline{L_n(1/\overline z)}=
1+z\sum_{\ell=0}^{n-1} \overline{L_{\ell+1}(0)}L_\ell(z), \quad z\ne 0.
\end{equation}
For $|z|=1$, we have
\begin{equation} \label{eq:3}
\left|\dfrac{L^*_{n+1}(z)}{L_n^*(z)}-1\right|=
\left|\dfrac{L_{n+1}(z)}{L_n(z)}-z\right|=|L_{n+1}(0)|,\quad n=0, 1,...
\end{equation}
For more details about these former identities see 
\cite{MR2378544, MR1018677, MR2105088,ZBMATH3477793,MR1544526}.

Along this work we are going to deal with the zeros 
of polynomials with complex coefficients, therefore it is 
convenient to state the following results.
\begin{thm}[Szeg{\H{o}}'s theorem \cite{MR1367960,MR1544526}]
Let $a(z)$, $b(z)$ and $c(z)$ the polynomials defined in Theorem \ref{thm:1.2}.
If all the zeros of $a(z)$ lie in a closed disk $\overline D$ and 
$\lambda_1$, ..., $\lambda_n$ are the zeros of $b(z)$, then every zero of 
$c(z)$ has the form $\lambda_\ell \gamma_\ell$, where $\gamma_\ell\in \overline D$.
\end{thm}
\begin{lem}[Cauchy's. Theorem (27,2) in  \cite{MR0225972}] 
If  $P(z)=a_n z^n+\dots +a_1z+a_0$ is a complex polynomial of degree at least one, 
then all the zeros of $P$ lie in a closed circle
\[
|z|\le 1+A,
\]
where $A=\max\{|a_0|,...,|a_{n-1}|\}/|a_n|$.
\end{lem}
\begin{lem}[Datt and Govil \cite{ZBMATH3610193}] 
If $P(z)=z^n+a_{n-1}z^{n-1}+\dots+a_0$ is a complex polynomial of degree at 
least one, then all the zeros of $P$ lie in a ring shaped region
\[
\dfrac{|a_0|}{2(1+B)^{n-1}(1+nB)}\le |z|\le 1+\lambda_0 B,
\]
where $B=\max\{|a_0|,...,|a_{n-1}|\}$, and $\lambda_0$ is the unique root of the 
equation $(x-1)(1+Bx)^n+1=0$ in the interval $[0,1]$.
\end{lem}
For more information about inequalities that satisfy the zeros of complex polynomials, 
read the survey \cite{MR1894720}.


Sendov’s conjecture \cite[p. 267]{MR1219199} asserts that if a  polynomial 
$f(z)$ of degree $n \ge  2$ 
has all of its zeroes in $\mathbb D$, then for each such zero $z_0$ there is a zero of the 
derivative $f'(z)$ in the $\overline{D(z_0,1)}$.

We are going to handle with $k$-polar polynomials. In the $k=1$ case it is 
straightforward to obtain 
\begin{equation} \label{eq:4}
(n+1)\int_{\xi}^z L_n(t) dt = (z-\xi)Q_{n;1}(z;\xi),
\end{equation} 
therefore it is logical to call $Q_{n;1}(z;\xi)$ the $n$-th first order 
polar polynomial of $L_n(z)$ (see \cite{MR1882659,MR2764236}). 
As a consequence of  \eqref{eq:4} we obtain
\begin{equation} \label{ eq:5}
(n+1)L_n(z)=Q_{n;1}(z;\xi)+(z-\xi)Q_{n;1}(z;\xi).
\end{equation}
\begin{rem}
Observe that, by construction, $Q_{n;1}(z;\xi)$ is a monic polynomial of degree $n$
and the pole of this polynomial is not irregular. In fact,
\[
\lim_{z\to \xi} Q_{n;1}(z)= \lim_{z\to \xi} \dfrac{\displaystyle (n+1)\int_\xi^z L_n(t)dt}{z-\xi}
=(n+1) L_n(\xi).
\]
Analogous calculations can be done in order to see that 
the $k$-polar monic polynomial $Q_{n;1}(z;\xi)$ has degree $n$
and the pole of this polynomial is not irregular.
\end{rem}
\section{Localization of zeros of polar polynomials}
We prove that all the zeros of the $k$-polar monic polynomial 
$Q_{n;k}(z;\xi)$ are contained in a disc whose radius is independent of $n$. 
First, let us express the polynomials $L_n(z)$ and  $Q_{n;k}(z;\xi)$ in terms of 
powers of $z-\xi$, that is
\begin{equation} \label{eq:9}
L_n(z)=\sum_{\ell=0}^n a_{n,\ell} (z-\xi)^\ell, \quad 
Q_{n;k}(z)=\sum_{\ell=0}^n b_{n,\ell;k} (z-\xi)^\ell,\quad k=1, 2, ..., 
\end{equation}
where $a_{n,n}=b_{n,n;k}=1$ for all $k=1, 2, ..$.
\begin{lem}\label{lem:3.1}
Let $k$ be a positive integer. Let $\xi\in \mathbb C$. 
The coefficients of $L_n(z)$ and $Q_{n;k}(z;\xi)$ are fulfill the relations
\begin{equation} \label{eq:10}
b_{n,\ell;k}=\dfrac{(n+k)\cdots (n+1)}{(\ell+k)\cdots (\ell+1)}\, a_{n,\ell},
\quad \ell=0, 1, ..., n-1.
\end{equation} 
\end{lem}
\begin{proof} 
By Definition \ref{def:1.3} we have
\begin{equation} \label{eq:11} 
\dfrac{ d^k}{dz^k} (z-\xi)^k Q_{n;k}(z;\xi)=(n+k)\cdots (n+1)L_n(z).
\end{equation} 
Start with \eqref{eq:11}, use the power expansion \eqref{eq:9} 
and take into account the linearity of the derivative. 
If we compare the power coefficients in these expressions the result holds.
\end{proof} 
By using this result we obtain the first main result.
\begin{thm}\label{thm:3.2} 
Let $\mu$ be a finite measure defined on the Borelian 
$\sigma$-algebra of $\mathbb C$ such that it contains an infinite 
number of points and let $(L_n(z))$ the system of monic orthogonal 
polynomials with respect to $\mu$.
Let $\xi$ be a fixed complex number. Let $k$ be a positive integer. 

All the zeros of $Q_{n;k}(z;\xi)$ are contained in the closed disk 
$\overline{D(0,|\xi|+(k+1)(1+|\xi|)}$.
\end{thm}
\begin{proof} 
Since $P_n$ is OPUC,  by \cite{MR2105088} we know the zeros of 
$L_n(z)$ lie in $\mathbb D$. Let us define $\omega:=z-\xi$ and, taking into 
account Lemma \ref{lem:3.1}, let us consider the polynomials
\[
f_n(\omega)=L_n(z)=\sum_{\ell=0}^n \dfrac{a_{n,\ell}}
{{n\choose \ell}}{n\choose \ell} \omega^\ell,
\]
and 
\[\begin{split}
g_{n;k}(\omega)&=\sum_{\ell=0}^n \dfrac{(n+k)\cdots (n+1)}
{(\ell+k)\cdots (\ell+1)}{n\choose \ell} \omega^\ell
=\sum_{\ell=0}^n {n+k\choose \ell+k} \omega^\ell
\\ &=\dfrac{(n+k)\cdots (k+1)}{n!}F(-n,1;k+1;-\omega),
\end{split}\]
where $F(a,b;c;z)$ is the Gau{\ss} function \cite[15.2.2]{NIST:DLMF}.

The “composition” of $f_n(z)$ and $g_n(z)$ leads to  
\[
h_{n;k}(\omega)=\sum_{\ell=0}^n \dfrac{b_{n,\ell;k}}
{{n\choose \ell}}{n\choose \ell} \omega^\ell= Q_{n;k}(z;\xi).
\]

Due to Theorem 3.2 in \cite{driver2002zeros} we know that $g_{n;k}(z)$ has 
non-real zeros for $n$ even and $n-1$ non-real zeros for $n$ odd. 
It is known this function is defined for $|\omega|<1$.

\begin{rem}
Observe that by using \cite[Corollary 2]{MR1931616} we obtain
\begin{equation} \label{eq:12}
(1-z)^{n+k}F(n+k+1,k;k+1;z)=F(-n,1;k+1;z)
=\dfrac{n!}{(k+1)\cdots (k+n)}P_n^{(k,-k-n)}(1-2z),
\end{equation} 
where $P_n^{(\alpha,\beta)}(z)$ is the Jacobi polynomial of degree 
$n$ and parameters $\alpha$ and $\beta$ (see, for example, \cite[18.5.7]{NIST:DLMF} 
we can express the polynomial $g_{n;k}(z)$ in terms of the Jacobi polynomials.
\end{rem}
\begin{rem}
Notice that if we take the $k$-th derivative of $\omega^k g_{n;k}(\omega)$ 
we obtain
\begin{equation} \label{eq:13}
\dfrac {d^k}{d\omega^k} \omega^k  g_{n;k}(\omega)=
\sum_{\ell=0}^n (n+k)\cdots (n+1)
{n\choose \ell} \omega^\ell=(n+k)\cdots (n+1)(1+\omega)^n.
\end{equation}
\end{rem}
Therefore for one side we know that if $z_{n,0}$ is a zero of $L_n(z)$ 
then $|z_{n,0}|\le 1$, so $|w_{n,0}|\le 1+|\xi|$. 
On the other hand,  we know that if $w_{n,1}$ is a root of 
$g_{n;k}(z)$ then $|w_{n,1}|\le R_{k,n}$. With these two 
inequalities we can claim that, using  Szeg{\H{o}}'s Theorem, 
for any root of $h_{n;k}(z)$, namely $|z_{n,3}|$, and since 
$|\omega_{n,3}|=|z_{n,3}+\xi|$, the following inequality holds:
\begin{equation} \label{eq:14} 
|z_{n,3}|\le (1+|\xi|)R_{n,k}+|\xi|.
\end{equation} 
Since $P_1^{(k,-k-n)}(1+\omega)=\omega+k+1$ and 
the zeros of the Jacobi polynomials $P_n^{(k,-k-n)}(1+\omega)$ 
tend to the circle $\overline{D(-1,1)}$ when $n\to \infty$ (see \cite{MR2149265}), 
we can assume $R_{k,n}\le k+1$. Hence the result follows.
\end{proof} 
\section{The examples}
We will consider different examples which let us to explain 
why sometimes we can consider a ring shape region (strictly speaking) where the 
zeros are located in, and some other situations we cannot. Of course the 
region depends on different parameters such as the value $\xi$, as well as 
the integers $k$ and $n$, among others. These examples can 
be consider as canonical.

{\bf First example. The Bernstein-Szeg{\H{o}} Polynomials}
Let $\beta\in \mathbb D$, and let us consider the measure 
\begin{equation} \label{eq:15}
d\mu(z)=\dfrac{1}{|z+\beta|^2}\dfrac{d\theta}{2\pi}.
\end{equation} 

Notice that if $\beta=r\exp(-i \phi)$, then \eqref{eq:15} becomes \cite[(1.6.2)]{MR2105088} 
\[
d\mu(z)=P_r(\theta,\phi)\dfrac{d\theta}{2\pi},
\]
where $P_r$ is the Poisson kernel of 
\[
\int \Re (g(z)) z^{-n}\, \dfrac{d\theta}{2\pi}.
\]
Let $(L_n(z))$ be the monic orthogonal polynomials with respect to $\mu$. 
These polynomials can be expressed as follows \cite{MR2556830, MR1239908}:
\begin{equation}\label{eq:16} 
L_n(z)=z^n+\beta z^{n-1},\quad n=1, 2, ....
\end{equation} 
From this expression their first order monic polar polynomials are defined as
\begin{equation} \label{eq:17}
Q_{n;1}(z;\xi)=(n+1)\dfrac{\displaystyle \int_\xi^z L_n(t) dt}{z-\xi}=
\dfrac{z^{n}(nz+(n+1)\beta)-\xi^{n}(n\xi+(n+1)\beta)}{n(z-\xi)}.
\end{equation}
The second order monic polar polynomial of degree $n$ is
\begin{equation} \label{eq:18}
\hspace{-10mm}Q_{n;2}(z;\xi)=\dfrac{z^{n+1}(n z+\beta n+2 \beta)+\xi ^n \left(n (n+1) \xi ^2+n(n+2) \xi 
(\beta-z)-\beta (n+1) (n+2) z\right)}{n (n+1) (n+2) (z-\xi)^2}.
\end{equation}
In Figure \ref{fig1} we  show the zeros of these polynomials under different settings.
\begin{figure}[!hbt]
\begin{center}
\begin{minipage}{0.46\textwidth}
\begin{tikzpicture}[domain=-0.7:0.7,font=\sffamily, scale=6]
\draw[-stealth] (-0.5,0) -- (0.6,0) node[above] {$x$};
\draw[-stealth] (0,-0.55) -- (0,0.55) node[left] {$y$};
\draw plot [only marks, mark=*, mark options={fill=black},mark size=0.3pt]
coordinates{(-0.3379, -0.02939) (-0.3374, 0.02357) (-0.3303, -0.08210)
(-0.3288, 0.07559) (-0.3145, -0.1334) (-0.3123, 0.1255) (-0.2908, 
-0.1822) (-0.2885, 0.1722) (-0.2594, -0.2273) (-0.2579, 0.2145)
(-0.2562, -0.4438) (-0.2212, 0.2516) (-0.2202, -0.2675) (-0.1793, 
0.2825) (-0.1736, -0.3007) (-0.1332, 0.3065) (-0.1216, -0.3245)
(-0.08399, 0.3232) (-0.06716, -0.3383) (-0.03286, 0.3321)
(-0.01235, -0.3430) (0.01898, 0.3330) (0.04173, -0.3391) (0.07032, 
0.3259) (0.09403, -0.3271) (0.1199, 0.3110) (0.1435, -0.3073)
(0.1891, -0.2804) (0.2094, 0.2594) (0.2298, -0.2469) (0.2471, 
0.2239) (0.2649, -0.2077) (0.2789, 0.1829) (0.2934, -0.1638)
(0.3040, 0.1376) (0.3147, -0.1161) (0.3219, 0.08887) (0.3284, 
-0.06576) (0.3321, 0.03793) (0.3343, -0.01402)}; 
\draw[gray] plot [only marks, mark=*, mark options={fill=gray},mark size=0.3pt]
coordinates{ (-0.3401, 0.007950) (-0.3362, -0.06303) (-0.3296, 0.07766) 
(-0.3178, -0.1326) (-0.3054, 0.1434) (-0.2851, -0.1981) (-0.2685, 
0.2025) (-0.2583, -0.4474) (-0.2380, -0.2570) (-0.2205, 0.2528) 
(-0.1760, -0.3049) (-0.1636, 0.2921) (-0.1036, -0.3346) (-0.09990,
0.3190) (-0.03222, 0.3323) (-0.02903, -0.3460) (0.03665, 0.3315) 
(0.04409, -0.3415) (0.1039, 0.3168) (0.1135, -0.3227) (0.1769, 
-0.2905) (0.2224, 0.2484) (0.2321, -0.2466) (0.2688, 0.1975) 
(0.2769, -0.1928) (0.3039, 0.1382) (0.3097, -0.1313) (0.3262, 
0.07298) (0.3292, -0.06460) (0.3348, 0.004421)};  
\draw  plot[only marks, mark=*, mark options={fill=white},mark size=0.3pt]
coordinates{(-0.343894212383155215, -0.0238367726899163517) 
(-0.3311833801254466726, 
  0.0817439042110546115) (-0.3242955991613638362, 
-0.1307602066248916952) (-0.2882224810856428162, 
  0.1774301993002356972) (-0.2715541484374780871, 
-0.2309808114417706562) (-0.2624511763208994692, 
-0.4545787718940157314) (-0.2193045321285863175, 
  0.2551982092988473708) (-0.1811142539804016368, 
-0.313699089868989404) (-0.1308585952341046769, 
  0.3084424609047654554) (-0.0642581762765776751, 
-0.3506631967707916775) (-0.0309718438839243065, 
  0.3326437904324377765) (0.0489061388394234584, 
-0.3462283305216824269) (0.0712897724579758331, 
  0.325810264630665313) (0.1513038554978749193, 
-0.3097395104832116719) (0.2363839877140183809, 
-0.245941268394307882) (0.2465150797548988203, 
  0.2246438862939519305) (0.2977703005353614512, 
-0.1608928909118287018) (0.3035638948676002677, 
  0.1394994916106947442) (0.3306603983173931384, 
-0.0623241913389926785) (0.3325483043663677732, 0.0408943626761028061)}; 
\draw[black] plot[only marks, mark=*, mark options={fill=gray},mark size=0.3pt] 
coordinates{(-0.343, -0.124) (-0.335, 0.093) (-0.271, -0.470) (-0.215, 0.262) 
(-0.198, -0.344) (-0.027, 0.333) (0.064, -0.359) (0.248, -0.243) (0.302, 0.143)
(0.334, -0.0557)};  
\begin{scope}
\draw[black] plot[mark=*, mark options={fill=gray}, mark size=0.3pt] (0.5,0.6)
node[right]{\color{black}\scriptsize $n=1$};
\draw plot[mark=*, mark options={fill=white},mark size=0.3pt] (0.5,0.5)
node[right]{\scriptsize $n=2$};
\draw[gray] plot[mark=*, mark options={fill=gray},mark size=0.3pt] (0.5,0.4)
node[right]{\scriptsize $n=3$};
\draw plot[mark=*, mark options={fill=black},mark size=0.3pt] (0.5,0.3)
node[right]{\scriptsize $n=4$};
\end{scope}
\foreach \x/\xtext in {0.33/\frac 13, 0.5/\frac 12}
\draw[shift={(\x,0)}] (0pt, 1pt) -- (0pt,-1pt) node[below] {\tiny $\xtext$};
\foreach \y/\ytext in {0.33/\frac 13, -0.5/-\frac 12}
\draw[shift={(0,\y)}] (-1pt,0pt) -- (1pt,0pt) node[right] {\tiny $\ytext$};
\end{tikzpicture}
\end{minipage}
\begin{minipage}{0.46\textwidth}
\begin{tikzpicture}[domain=-0.7:0.7,font=\sffamily, scale=6]
\draw[-stealth] (-0.5,0) -- (0.6,0) node[above] {$x$};
\draw[-stealth] (0,-0.55) -- (0,0.55) node[left] {$y$};
\draw plot [only marks, mark=*, mark options={fill=black},mark size=0.3pt]
coordinates{(-0.3751, -0.03613) (-0.3747, 0.02195) (-0.3667, -0.09396) 
(-0.3655, 0.07897) (-0.3495, -0.1503) (-0.3477, 0.1337) (-0.3236, 
-0.2040) (-0.3219, 0.1848) (-0.2894, -0.2540) (-0.2887, 0.2311) 
(-0.2625, -0.4547) (-0.2489, 0.2716) (-0.2465, -0.2991) (-0.2035, 
0.3054) (-0.1945, -0.3368) (-0.1536, 0.3315) (-0.1358, -0.3630) 
(-0.1004, 0.3495) (-0.07527, -0.3776) (-0.04523, 0.3587) 
(-0.01486, -0.3823) (0.01045, 0.3590) (0.04456, -0.3778) (0.06519, 
0.3498) (0.1020, -0.3646) (0.1174, 0.3306) (0.1563, -0.3429) 
(0.2064, -0.3135) (0.2276, 0.2670) (0.2511, -0.2770) (0.2704, 
0.2314) (0.2896, -0.2343) (0.3056, 0.1885) (0.3210, -0.1864) 
(0.3333, 0.1402) (0.3445, -0.1345) (0.3528, 0.08782) (0.3597, 
-0.07973) (0.3639, 0.03278) (0.3662, -0.02353)}; 
\draw[gray] plot [only marks, mark=*, mark options={fill=gray},mark size=0.3pt]
coordinates{(-0.3872, 0.003888) (-0.3827, -0.07561) (-0.3758, 0.08189) 
(-0.3623, -0.1536) (-0.3491, 0.1554) (-0.3261, -0.2275) (-0.3083, 
0.2215) (-0.2741, -0.2951) (-0.2667, -0.4619) (-0.2554, 0.2775) 
(-0.2032, -0.3520) (-0.1925, 0.3213) (-0.1223, 0.3508) (-0.1185, 
-0.3849) (-0.04807, 0.3649) (-0.03396, -0.3962) (0.02702, 0.3625) 
(0.04810, -0.3906) (0.09929, 0.3426) (0.1259, -0.3693) (0.1970, 
-0.3334) (0.2471, 0.2573) (0.2588, -0.2845) (0.3005, 0.2047) 
(0.3091, -0.2246) (0.3400, 0.1408) (0.3460, -0.1563) (0.3650, 
0.06947) (0.3680, -0.08238) (0.3744, -0.006046)};  
\draw  plot[only marks, mark=*, mark options={fill=white},mark size=0.3pt]
coordinates{(-0.4091, -0.03467) (-0.3946, 0.08774) (-0.3869, -0.1590)
(-0.3456, 0.1985) (-0.3279, -0.2769) (-0.2743, -0.4751) (-0.2670, 
0.2883) (-0.2218, -0.3842) (-0.1663, 0.3492) (-0.07584, -0.4225)
(-0.05318, 0.3756) (0.05577, -0.4145) (0.06106, 0.3635) (0.1745, 
-0.3716) (0.2733, -0.2979) (0.2843, 0.2346) (0.3447, -0.2000)
(0.3519, 0.1417) (0.3831, -0.08706) (0.3855, 0.03061)}; 
\draw[black] plot[only marks, mark=*, mark options={fill=gray},mark size=0.3pt] 
coordinates{(-0.4550, -0.1669) (-0.4423, 0.1058) (-0.3406, -0.4629) (-0.2947, 
0.3153) (-0.2306, -0.5264) (-0.06540, 0.4001) (0.08295, -0.4775) 
(0.3128, -0.3301) (0.3792, 0.1434) (0.4204, -0.09756)};  
\begin{scope}
\draw[black] plot[mark=*, mark options={fill=gray}, mark size=0.3pt] (0.5,0.6)
node[right]{\color{black}\scriptsize $n=1$};
\draw plot[mark=*, mark options={fill=white},mark size=0.3pt] (0.5,0.5)
node[right]{\scriptsize $n=2$};
\draw[gray] plot[mark=*, mark options={fill=gray},mark size=0.3pt] (0.5,0.4)
node[right]{\scriptsize $n=3$};
\draw plot[mark=*, mark options={fill=black},mark size=0.3pt] (0.5,0.3)
node[right]{\scriptsize $n=4$};
\end{scope}
\foreach \x/\xtext in {0.33/\frac 13, 0.5/\frac 12}
\draw[shift={(\x,0)}] (0pt, 1pt) -- (0pt,-1pt) node[below] {\tiny $\xtext$};
\foreach \y/\ytext in {0.33/\frac 13, -0.5/-\frac 12}
\draw[shift={(0,\y)}] (-1pt,0pt) -- (1pt,0pt) node[right] {\tiny $\ytext$};
\end{tikzpicture}
\end{minipage}
\caption{\label{fig1} Left: Zeros of $Q_{10n;1}(z;1/3 \exp(i\pi/3))$ for 
$n=1, 2, 3, 4$ in the
window $[-0.5,0.7]\times[-0.5,0.7]$, with $\beta=1/2\exp(i\pi/3)$. Right:
Zeros of $Q_{10n;2}(z;1/3 \exp(i\pi/3))$ for $n= 1, 2, 3, 4$ 
in the window $[-0.5,0.7]\times[-0.5,0.7]$, with $\beta=1/2\exp(i\pi/3)$.}
\end{center}
\end{figure}
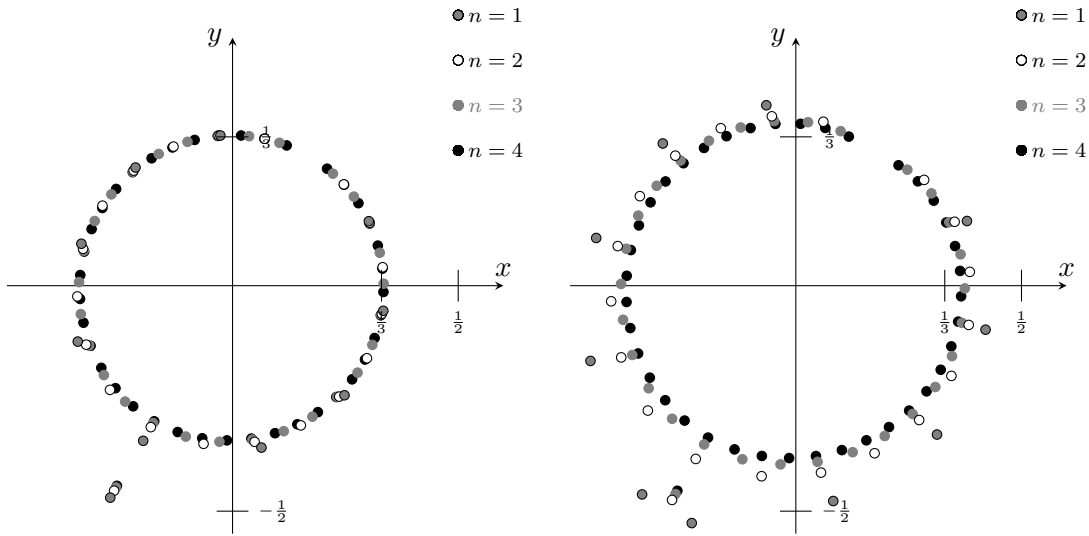 

In order to show that Sendov's conjecture remains valid, in Table \ref{tab1-2}
we present the maximum distance between each zero of $Q_{n,1}(z,\xi)$ 
and the closet zero of $Q'_{n,1}(z,\xi)$ for different values of $n$. 

\begin{table}[!hbt]
\begin{minipage}{0.46\textwidth}
\begin{center}
\begin{tabular}{c|c|l}
$n$& zero& distance\\
\hline \hline 
2&$-0.04549-0.59920 i$&0.2602\\
3&$0.25331-0.28868 i$&0.41997\\
4&$0.32559-0.13863 i$&0.65268\\
5&$0.34111-0.04380 i$&0.68222\\
6&$0.33895+0.01995i$&0.74277\\
7&$0.33098+0.06499i$&0.75514\\
8&$0.32137+0.09813i$&0.77831\\
9&$-0.04905+0.33164i$&0.78565\\
10&$-0.02743+0.33361i$&0.79638\\
11&$-0.00958+0.33426i$&0.80117\\
12&$0.00536+0.33412i$&0.80687\\
13&$0.01804+0.33350i$&0.81013\\
14&$0.02892+0.33261i$&0.81350\\
15&$0.03835+0.33155i$&0.81577\\
16&$0.04660+0.33042i$&0.81795\\
17&$0.05388+0.32925i$&0.81957\\
18&$0.06033+0.32808i$&0.82106\\
19&$0.06610+0.32693i$&0.82225\\
20&$0.07129+0.32581i$&0.82332\end{tabular}
\end{center}
\end{minipage}
\hspace{10mm}
\begin{minipage}{0.46\textwidth}
\begin{center}
\begin{tabular}{c|c|l}
$n$& zero& distance\\
\hline \hline 
2&$-0.89538-0.44530i$&0.5528\\
3&$0.40206-0.49827i$&0.85603\\
4&$0.47459-0.25270i$&0.89688\\
5&$-0.57262-0.15872i$&0.85890\\
6&$0.45849-0.01783i$&0.90690\\
7&$-0.35212-0.60990i$&0.88012\\
8&$0.41566+0.08718i$&0.90725\\
9&$-0.32375-0.56075i$&0.88470\\
10&$-0.06540+0.40007i$&0.90141\\
11&$-0.04147+0.39636i$&0.88964\\
12&$-0.02178+0.39227i$&0.89539\\
13&$-0.00532+0.38809i$&0.88893\\
14&$0.00861+0.38400i$&0.89015\\
15&$0.02054+0.38006i$&0.88620\\
16&$0.03086+0.37633i$&0.88568\\
17&$0.03986+0.37280i$&0.88297\\
18&$0.04779+0.36949i$&0.88182\\
19&$0.05481+0.36639i$&0.87977\\
20&$0.06106+0.36348i$&0.87845
\end{tabular}
\end{center}
\end{minipage}
\caption{\label{tab1-2}For every  $2\le n\le 20$, $\beta=1/2\exp(i\pi/3)$, 
we obtain the zero of  $Q_{n;1}(z;1/3 \exp(i\pi/3))$, left, and 
$Q_{n;2}(z;1/3 \exp(i\pi/3))$, right, which produces the 
maximum distance with respect to the zeros of their corresponding derivatives.}
\end{table}
\newpage 
{\bf Second example} Fix $m\ge 0$. Let 
\[
d \mu_1(\theta)=\frac {d\theta}{2\pi}+m\delta(z-1),
\]
where $z=\exp(i\theta)$ and 
\[
\int f(z)\, \delta(z-1) d\theta = f(1),\quad f\in \mathbb P[z].
\]
The monic associated orthogonal polynomials with respect to 
$\mu_1$ on $\mathbb T$ are \cite{MR1367960, MR2556830}
\begin{equation} \label{eq:19}
L_n(z)=z^n-\dfrac{m}{1+nm}\sum_{k=0}^{n-1} z^k=z^n-\dfrac{m}
{1+nm}\dfrac{z^n-1}{z-1},\quad n=1,2, ...
\end{equation} 
From this expression their first order monic polar polynomials are defined as
\begin{equation} \label{eq:20}
Q_{n;1}(z;\xi)=
\dfrac{z^{n+1}-\xi^{n+1}}{z-\xi}-\dfrac{m(n+1)}{1+nm}\sum_{k=0}^{n-1}\dfrac{z^{k+1}-\xi^{k+1}}
{(k+1)(z-\xi)}.
\end{equation}
\begin{rem}
The zeros of these polynomials tend to accumulate 
around the unit circle  $\mathbb T$ whenever $|\xi|\le 1$, and 
around the closed circle $\overline{D(0,|\xi|)}$ 
whenever $|\xi|>1$. 
Since we added a mass point at $z=1$, it is 
expected that some of zeros the polar polynomials close to $z=1$
move outside of such boundary.

Due Theorem \ref{thm:3.2} we know all the zeros of 
these polynomials lie inside of 
$\overline{D(0,2+3|\xi|)}$. 
\end{rem} 

In Figure \ref{fig2} we show the zeros of these polynomials 
under different settings (in the first case $|\xi_1|+ 2(1 +|\xi_1|)=3$
 and $|\xi_2|+2(1 +|\xi_2|)=6$ in the second one).

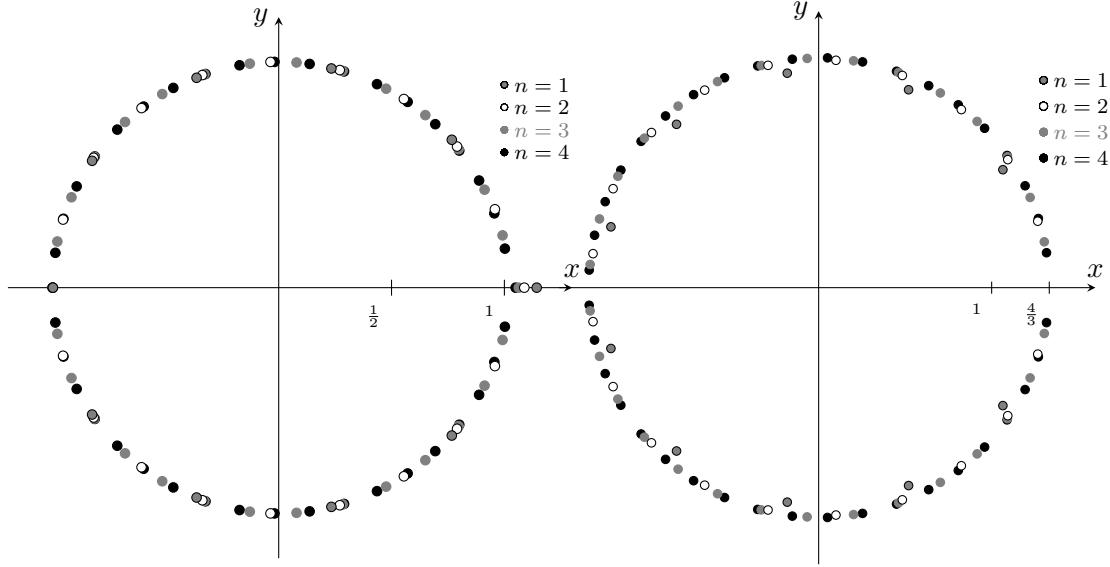
\begin{figure}[!hbt]
\begin{center}
\begin{minipage}{0.45\textwidth}
\begin{tikzpicture}[domain=-1.2:1.2,font=\sffamily, scale=3]
\draw[-stealth] (-1.2,0) -- (1.3,0) node[above] {$x$};
\draw[-stealth] (0,-1.2) -- (0,1.2) node[left] {$y$};
\draw plot [only marks, mark=*, mark options={fill=black},mark size=0.6pt]
coordinates{(-1.0017, 0) (1.0498, 0) (-0.98971, -0.15464) (-0.98971, 
  0.15464) (-0.95395, -0.30558) (-0.95395, 
  0.30558) (-0.89530, -0.44920) (-0.89530, 
  0.44920) (-0.81515, -0.58206) (-0.81515, 
  0.58206) (-0.71541, -0.70098) (-0.71541, 
  0.70098) (-0.59844, -0.80309) (-0.59844, 
  0.80309) (-0.46704, -0.88593) (-0.46704, 
  0.88593) (-0.32431, -0.94749) (-0.32431, 
  0.94749) (-0.17366, -0.98627) (-0.17366, 
  0.98627) (-0.018663, -1.0013) (-0.018663, 
  1.0013) (0.13699, -0.99214) (0.13699, 
  0.99214) (0.28959, -0.95895) (0.28959, 
  0.95895) (0.43553, -0.90245) (0.43553, 
  0.90245) (0.57135, -0.82387) (0.57135, 
  0.82387) (0.69386, -0.72495) (0.69386, 
  0.72495) (0.80020, -0.60787) (0.80020, 
  0.60787) (0.88800, -0.47517) (0.88800, 
  0.47517) (0.95552, -0.32953) (0.95552, 
  0.32953) (1.0022, -0.17315) (1.0022, 0.17315)}; 
\draw[gray] plot [only marks, mark=*, mark options={fill=gray},mark size=0.6pt]
coordinates{(-1.0020, 0) (1.0636, 0) (-0.98088, -0.20473) (-0.98088, 
  0.20473) (-0.91831, -0.40083) (-0.91831, 
  0.40083) (-0.81693, -0.58005) (-0.81693, 
  0.58005) (-0.68096, -0.73482) (-0.68096, 
  0.73482) (-0.51608, -0.85858) (-0.51608, 
  0.85858) (-0.32916, -0.94606) (-0.32916, 
  0.94606) (-0.12801, -0.99345) (-0.12801, 
  0.99345) (0.078993, -0.99862) (0.078993, 
  0.99862) (0.28320, -0.96116) (0.28320, 
  0.96116) (0.47612, -0.88237) (0.47612, 
  0.88237) (0.64975, -0.76521) (0.64975, 
  0.76521) (0.79702, -0.61408) (0.79702, 
  0.61408) (0.91222, -0.43448) (0.91222, 
  0.43448) (0.99196, -0.23196) (0.99196, 0.23196)};  
\draw  plot[only marks, mark=*, mark options={fill=white},mark size=0.6pt]
coordinates{(-1.0023, 0) (1.0883, 0) (-0.95568, -0.30195) (-0.95568, 
  0.30195) (-0.82009, -0.57592) (-0.82009, 
  0.57592) (-0.60792, -0.79643) (-0.60792, 
  0.79643) (-0.33851, -0.94283) (-0.33851, 
  0.94283) (-0.036502, -1.0011) (-0.036502, 
  1.0011) (0.27054, -0.96499) (0.27054, 
  0.96499) (0.55464, -0.83663) (0.55464, 
  0.83663) (0.79029, -0.62579) (0.79029, 
  0.62579) (0.95797, -0.34771) (0.95797, 0.34771)}; 
\draw[black] plot[only marks, mark=*, mark options={fill=gray},mark size=0.6pt] 
coordinates{(-1.0005, 0) (1.1433, 0) (-0.82687, -0.56264) (-0.82687, 
  0.56264) (-0.36399, -0.93073) (-0.36399, 
  0.93073) (0.23348, -0.97212) (0.23348, 
  0.97212) (0.76712, -0.65585) (0.76712, 0.65585)};  
\begin{scope}
\draw[black] plot[mark=*, mark options={fill=gray}, mark size=0.4pt] (1,0.9)
node[right]{\color{black}\scriptsize $n=1$};
\draw plot[mark=*, mark options={fill=white},mark size=0.4pt] (1,0.8)
node[right]{\scriptsize $n=2$};
\draw[gray] plot[mark=*, mark options={fill=gray},mark size=0.4pt] (1,0.7)
node[right]{\scriptsize $n=3$};
\draw plot[mark=*, mark options={fill=black},mark size=0.4pt] (1,0.6)
node[right]{\scriptsize $n=4$};
\end{scope}
\foreach \x/\xtext in {1/1, 0.5/\frac 12}
\draw[shift={(\x,0)}] (0pt, 1pt) -- (0pt,-1pt) node[below left] {\tiny $\xtext$};
\end{tikzpicture}
\end{minipage}
\begin{minipage}{0.46\textwidth}
\begin{tikzpicture}[domain=-0.7:0.7,font=\sffamily, scale=2.3]
\draw[-stealth] (-1.5,0) -- (1.6,0) node[above] {$x$};
\draw[-stealth] (0,-1.6) -- (0,1.6) node[left] {$y$};
\draw plot [only marks, mark=*, mark options={fill=black},mark size=0.7pt]
coordinates{(-1.3263, -0.10184) (-1.3263, 
  0.10184) (-1.2952, -0.30314) (-1.2952, 
  0.30314) (-1.2338, -0.49734) (-1.2338, 
  0.49734) (-1.1434, -0.67987) (-1.1434, 
  0.67987) (-1.0262, -0.84647) (-1.0262, 
  0.84647) (-0.88493, -0.99322) (-0.88493, 
  0.99322) (-0.72291, -1.1167) (-0.72291, 
  1.1167) (-0.54395, -1.2140) (-0.54395, 
  1.2140) (-0.35222, -1.2828) (-0.35222, 
  1.2828) (-0.15223, -1.3215) (-0.15223, 
  1.3215) (0.051342, -1.3293) (0.051342, 
  1.3293) (0.25372, -1.3059) (0.25372, 
  1.3059) (0.45017, -1.2519) (0.45017, 
  1.2519) (0.63609, -1.1685) (0.63609, 
  1.1685) (0.80711, -1.0576) (0.80711, 
  1.0576) (0.95926, -0.92200) (0.95926, 
  0.92200) (1.0890, -0.76470) (1.0890, 
  0.76470) (1.1933, -0.58942) (1.1933, 
  0.58942) (1.2699, -0.40030) (1.2699, 
  0.40030) (1.3171, -0.20204) (1.3171, 0.20204)}; 
\draw[gray] plot [only marks, mark=*, mark options={fill=gray},mark size=0.7pt]
coordinates{(-1.3208, -0.13435) (-1.3208, 
  0.13435) (-1.2667, -0.39754) (-1.2667, 
  0.39754) (-1.1607, -0.64445) (-1.1607, 
  0.64445) (-1.0072, -0.86497) (-1.0072, 
  0.86497) (-0.81240, -1.0501) (-0.81240, 
  1.0501) (-0.58433, -1.1921) (-0.58433, 
  1.1921) (-0.33230, -1.2854) (-0.33230, 
  1.2854) (-0.066633, -1.3260) (-0.066633, 
  1.3260) (0.20179, -1.3123) (0.20179, 
  1.3123) (0.46200, -1.2448) (0.46200, 
  1.2448) (0.70335, -1.1263) (0.70335, 
  1.1263) (0.91599, -0.96156) (0.91599, 
  0.96156) (1.0913, -0.75732) (1.0913, 
  0.75732) (1.2223, -0.52189) (1.2223, 
  0.52189) (1.3042, -0.26517) (1.3042, 0.26517)};  
\draw  plot[only marks, mark=*, mark options={fill=white},mark size=0.7pt]
coordinates{(-1.3041, -0.19665) (-1.3041, 
  0.19665) (-1.1881, -0.57248) (-1.1881, 
  0.57248) (-0.96649, -0.89742) (-0.96649, 
  0.89742) (-0.65885, -1.1426) (-0.65885, 
  1.1426) (-0.29255, -1.2861) (-0.29255, 
  1.2861) (0.099901, -1.3153) (0.099901, 
  1.3153) (0.48367, -1.2274) (0.48367, 
  1.2274) (0.82475, -1.0301) (0.82475, 
  1.0301) (1.0931, -0.74067) (1.0931, 
  0.74067) (1.2664, -0.38437) (1.2664, 0.38437)}; 
\draw[black] plot[only marks, mark=*, mark options={fill=gray},mark size=0.7pt] 
coordinates{(-1.2007, -0.35180) (-1.2007, 
  0.35180) (-0.82147, -0.94486) (-0.82147, 
  0.94486) (-0.18140, -1.2407) (-0.18140, 
  1.2407) (0.51981, -1.1448) (0.51981, 
  1.1448) (1.0649, -0.68221) (1.0649, 0.68221)};  
\begin{scope}
\draw[black] plot[mark=*, mark options={fill=gray}, mark size=0.6pt] (1.3,1.2)
node[right]{\color{black}\scriptsize $n=1$};
\draw plot[mark=*, mark options={fill=white},mark size=0.65pt] (1.3,1.05)
node[right]{\scriptsize $n=2$};
\draw[gray] plot[mark=*, mark options={fill=gray},mark size=0.6pt] (1.3,0.9)
node[right]{\scriptsize $n=3$};
\draw plot[mark=*, mark options={fill=black},mark size=0.6pt] (1.3,0.75)
node[right]{\scriptsize $n=4$};
\end{scope}
\foreach \x/\xtext in {1/1, 1.333/\frac 43}
\draw[shift={(\x,0)}] (0pt, 1pt) -- (0pt,-1pt) node[below left] {\tiny $\xtext$};
\end{tikzpicture}
\end{minipage}
\caption{\label{fig2} Left: Zeros of $Q_{10n;1}(z;1/3)$ for $n=1, 2, 3, 4$ 
in the window $[-1,1.2]\times[-1,1]$, with $m=2/3$. 
Right: Zeros of $Q_{10n;1}(z;4/3)$ for $n=1,2,3,4$ in the window 
$[-1.5,1.5]\times[-1.5,1.5]$, with $m=2/3$.}
\end{center}
\end{figure} 
In order to show that Sendov's conjecture remains valid in Table 
\ref{tab3-4} we present the maximum distance between each zero of $Q_{n,1}(z,\xi)$ 
and the closet zero of $Q'_{n,1}(z,\xi)$ for different values of $n$. 
\begin{table}[!hbt]
\begin{minipage}{0.46\textwidth}
\begin{center}
\begin{tabular}{c|c|l}
$n$& zero& distance\\
\hline \hline 
2&$0.99163$&0.9440\\
3&$1.1388$&1.5309\\
4&$-0.97769$&1.5490\\
5&$1.1848$&1.7297\\
6&$-0.99320$&1.7641\\
7&$1.1711$&1.7968\\
8&$-0.99834$&1.8563\\
9&$-0.94557-0.32389i$&1.8579\\
10&$-1.0005$&1.9036\\
11&$-0.96404-0.26952i$&1.9012\\
12&$-1.0015$&1.9310\\
13&$-0.97483-0.23056i$&1.9275\\
14&$-1.0020$&1.9483\\
15&$-0.98163-0.20134i$&1.9446\\
16&$-1.0022$&1.9598\\
17&$-0.98617-0.17864i$&1.9563\\
18&$-1.0023$&1.9679\\
19&$-0.98934-0.16051i$&1.9647\\
20&$-1.0023$&1.9738
\end{tabular}
\end{center}
\end{minipage}
\hspace{10mm}
\begin{minipage}{0.46\textwidth}
\begin{center}
\begin{tabular}{c|c|l}
$n$& zero& distance\\
\hline \hline 
2&$-0.45238+0.38021i$&0.38021\\
3&$-0.24822+0.77992i$&1.2234\\
4&$0.13378+0.95290i$&1.6349\\
5&$0.44282+0.96439i$&1.7519\\
6&$-0.28585-1.04862i$&1.9307\\
7&$-0.81931-0.79845i$&2.0409\\
8&$-0.60642-1.03052i$&2.1514\\
9&$-0.99493-0.71892i$&2.2245\\
10&$-0.82147-0.94486i$&2.2841\\
11&$-1.09907-0.63389i$&2.3315\\
12&$-0.95942-0.84975i$&2.3671\\
13&$-1.16310-0.56021i$&2.3993\\
14&$-1.05017-0.76336i$&2.4222\\
15&$-1.20452-0.49917i$&2.4452\\
16&$-1.11214-0.68904i$&2.4609\\
17&$-1.23260-0.44888i$&2.4779\\
18&$-1.15604-0.62599i$&2.4892\\
19&$-1.25243-0.40715i$&2.5022\\
20&$-1.18814-0.57248i$&2.5107
\end{tabular}
\end{center}
\end{minipage}
\caption{\label{tab3-4}For every  $2\le n\le 20$, $m=2/3$, we obtain the zero of 
$Q_{n;1}(z;1/3)$, left, and $Q_{n;1}(z;4/3 )$, right, which produces the 
maximum distance with respect to the zeros of their corresponding derivatives.}
\end{table}

{\bf Third Example}
Let 
\[
d \mu_2(\theta)=|z-1|^2\dfrac{d\theta}{2\pi}.
\]
The monic associated orthogonal polynomials with respect to 
$\mu_2$ on $\mathbb T$ are \cite{MR1367960, MR2556830}
\begin{equation} \label{eq:21}
L_n(z)=\sum_{k=0}^{n} \dfrac{k+1}{n+1}z^k=\dfrac{(n+1)z^{n+2}-(n+2)z^{n+1}+1}{(n+1)(z-1)^2},\quad n=1,2, ...
\end{equation} 
From this expression their first order monic polar polynomials are defined as
\begin{equation} \label{eq:22}
Q_{n;1}(z;\xi)=\dfrac{z\left(z^{n+1}-1\right)(\xi-1)-\xi\left(\xi^{n+1}-1\right)(z-1)}{(\xi-1)(z-\xi)(z-1)},\quad z\ne 1, z\ne \xi.
\end{equation}

In Figure \ref{fig3} we show the zeros for the $k=1$ and $k=4$ polar polynomials. 
In Figure \ref{fig3} we show  the zeros of these polynomials under different 
settings (in the first case $|\xi_1|+2(1 +|\xi_1|)=3$ and
$|\xi_2|+5(1 +|\xi_2|)=13$ in the second one).

In the $k=4$ case, and due of the length and difficulty of these expressions for the polynomials,  are not presented in the manuscript.

\begin{figure}[!hbt]
\begin{minipage}{0.45\textwidth}
\begin{center}
\begin{tikzpicture}[domain=-1.2:1.2,font=\sffamily, scale=3]
\draw[-stealth] (-1.2,0) -- (1.3,0) node[above] {$x$};
\draw[-stealth] (0,-1.2) -- (0,1.2) node[left] {$y$};
\draw plot [only marks, mark=*, mark options={fill=black},mark size=0.6pt]
coordinates{(-1.0017, 0) (1.0498, 0) (-0.98971, -0.15464) (-0.98971, 
  0.15464) (-0.95395, -0.30558) (-0.95395, 
  0.30558) (-0.89530, -0.44920) (-0.89530, 
  0.44920) (-0.81515, -0.58206) (-0.81515, 
  0.58206) (-0.71541, -0.70098) (-0.71541, 
  0.70098) (-0.59844, -0.80309) (-0.59844, 
  0.80309) (-0.46704, -0.88593) (-0.46704, 
  0.88593) (-0.32431, -0.94749) (-0.32431, 
  0.94749) (-0.17366, -0.98627) (-0.17366, 
  0.98627) (-0.018663, -1.0013) (-0.018663, 
  1.0013) (0.13699, -0.99214) (0.13699, 
  0.99214) (0.28959, -0.95895) (0.28959, 
  0.95895) (0.43553, -0.90245) (0.43553, 
  0.90245) (0.57135, -0.82387) (0.57135, 
  0.82387) (0.69386, -0.72495) (0.69386, 
  0.72495) (0.80020, -0.60787) (0.80020, 
  0.60787) (0.88800, -0.47517) (0.88800, 
  0.47517) (0.95552, -0.32953) (0.95552, 
  0.32953) (1.0022, -0.17315) (1.0022, 0.17315)}; 
\draw[gray] plot [only marks, mark=*, mark options={fill=gray},mark size=0.6pt]
coordinates{(-1.0020, 0) (1.0636, 0) (-0.98088, -0.20473) (-0.98088, 
  0.20473) (-0.91831, -0.40083) (-0.91831, 
  0.40083) (-0.81693, -0.58005) (-0.81693, 
  0.58005) (-0.68096, -0.73482) (-0.68096, 
  0.73482) (-0.51608, -0.85858) (-0.51608, 
  0.85858) (-0.32916, -0.94606) (-0.32916, 
  0.94606) (-0.12801, -0.99345) (-0.12801, 
  0.99345) (0.078993, -0.99862) (0.078993, 
  0.99862) (0.28320, -0.96116) (0.28320, 
  0.96116) (0.47612, -0.88237) (0.47612, 
  0.88237) (0.64975, -0.76521) (0.64975, 
  0.76521) (0.79702, -0.61408) (0.79702, 
  0.61408) (0.91222, -0.43448) (0.91222, 
  0.43448) (0.99196, -0.23196) (0.99196, 0.23196)};  
\draw  plot[only marks, mark=*, mark options={fill=white},mark size=0.6pt]
coordinates{(-1.0023, 0) (1.0883, 0) (-0.95568, -0.30195) (-0.95568, 
  0.30195) (-0.82009, -0.57592) (-0.82009, 
  0.57592) (-0.60792, -0.79643) (-0.60792, 
  0.79643) (-0.33851, -0.94283) (-0.33851, 
  0.94283) (-0.036502, -1.0011) (-0.036502, 
  1.0011) (0.27054, -0.96499) (0.27054, 
  0.96499) (0.55464, -0.83663) (0.55464, 
  0.83663) (0.79029, -0.62579) (0.79029, 
  0.62579) (0.95797, -0.34771) (0.95797, 0.34771)}; 
\draw[black] plot[only marks, mark=*, mark options={fill=gray},mark size=0.6pt] 
coordinates{(-1.0005, 0) (1.1433, 0) (-0.82687, -0.56264) (-0.82687, 
  0.56264) (-0.36399, -0.93073) (-0.36399, 
  0.93073) (0.23348, -0.97212) (0.23348, 
  0.97212) (0.76712, -0.65585) (0.76712, 0.65585)};  
\begin{scope}
\draw[black] plot[mark=*, mark options={fill=gray}, mark size=0.4pt] (1,0.9)
node[right]{\color{black}\scriptsize $n=1$};
\draw plot[mark=*, mark options={fill=white},mark size=0.4pt] (1,0.8)
node[right]{\scriptsize $n=2$};
\draw[gray] plot[mark=*, mark options={fill=gray},mark size=0.4pt] (1,0.7)
node[right]{\scriptsize $n=3$};
\draw plot[mark=*, mark options={fill=black},mark size=0.4pt] (1,0.6)
node[right]{\scriptsize $n=4$};
\end{scope}
\foreach \x/\xtext in {1/1, 0.5/\frac 12}
\draw[shift={(\x,0)}] (0pt, 1pt) -- (0pt,-1pt) node[below left] {\tiny $\xtext$};
\end{tikzpicture}
\caption{\label{fig3} Zeros of $Q_{10n;1}(z;1/3)$ for $n=1, 2, 3, 4$ in the
window $[-1,1.2]\times[-1,1]$, with $m=2/3$.}
\end{center}
\end{minipage}
\hspace{1cm}
\begin{minipage}{0.46\textwidth}
\begin{center}
\begin{tikzpicture}[domain=-0.7:0.7,font=\sffamily, scale=2.3]
\draw[-stealth] (-1.5,0) -- (1.6,0) node[above] {$x$};
\draw[-stealth] (0,-1.6) -- (0,1.6) node[left] {$y$};
\draw plot [only marks, mark=*, mark options={fill=black},mark size=0.7pt]
coordinates{(-1.3263, -0.10184) (-1.3263, 
  0.10184) (-1.2952, -0.30314) (-1.2952, 
  0.30314) (-1.2338, -0.49734) (-1.2338, 
  0.49734) (-1.1434, -0.67987) (-1.1434, 
  0.67987) (-1.0262, -0.84647) (-1.0262, 
  0.84647) (-0.88493, -0.99322) (-0.88493, 
  0.99322) (-0.72291, -1.1167) (-0.72291, 
  1.1167) (-0.54395, -1.2140) (-0.54395, 
  1.2140) (-0.35222, -1.2828) (-0.35222, 
  1.2828) (-0.15223, -1.3215) (-0.15223, 
  1.3215) (0.051342, -1.3293) (0.051342, 
  1.3293) (0.25372, -1.3059) (0.25372, 
  1.3059) (0.45017, -1.2519) (0.45017, 
  1.2519) (0.63609, -1.1685) (0.63609, 
  1.1685) (0.80711, -1.0576) (0.80711, 
  1.0576) (0.95926, -0.92200) (0.95926, 
  0.92200) (1.0890, -0.76470) (1.0890, 
  0.76470) (1.1933, -0.58942) (1.1933, 
  0.58942) (1.2699, -0.40030) (1.2699, 
  0.40030) (1.3171, -0.20204) (1.3171, 0.20204)}; 
\draw[gray] plot [only marks, mark=*, mark options={fill=gray},mark size=0.7pt]
coordinates{(-1.3208, -0.13435) (-1.3208, 
  0.13435) (-1.2667, -0.39754) (-1.2667, 
  0.39754) (-1.1607, -0.64445) (-1.1607, 
  0.64445) (-1.0072, -0.86497) (-1.0072, 
  0.86497) (-0.81240, -1.0501) (-0.81240, 
  1.0501) (-0.58433, -1.1921) (-0.58433, 
  1.1921) (-0.33230, -1.2854) (-0.33230, 
  1.2854) (-0.066633, -1.3260) (-0.066633, 
  1.3260) (0.20179, -1.3123) (0.20179, 
  1.3123) (0.46200, -1.2448) (0.46200, 
  1.2448) (0.70335, -1.1263) (0.70335, 
  1.1263) (0.91599, -0.96156) (0.91599, 
  0.96156) (1.0913, -0.75732) (1.0913, 
  0.75732) (1.2223, -0.52189) (1.2223, 
  0.52189) (1.3042, -0.26517) (1.3042, 0.26517)};  
\draw  plot[only marks, mark=*, mark options={fill=white},mark size=0.7pt]
coordinates{(-1.3041, -0.19665) (-1.3041, 
  0.19665) (-1.1881, -0.57248) (-1.1881, 
  0.57248) (-0.96649, -0.89742) (-0.96649, 
  0.89742) (-0.65885, -1.1426) (-0.65885, 
  1.1426) (-0.29255, -1.2861) (-0.29255, 
  1.2861) (0.099901, -1.3153) (0.099901, 
  1.3153) (0.48367, -1.2274) (0.48367, 
  1.2274) (0.82475, -1.0301) (0.82475, 
  1.0301) (1.0931, -0.74067) (1.0931, 
  0.74067) (1.2664, -0.38437) (1.2664, 0.38437)}; 
\draw[black] plot[only marks, mark=*, mark options={fill=gray},mark size=0.7pt] 
coordinates{(-1.2007, -0.35180) (-1.2007, 
  0.35180) (-0.82147, -0.94486) (-0.82147, 
  0.94486) (-0.18140, -1.2407) (-0.18140, 
  1.2407) (0.51981, -1.1448) (0.51981, 
  1.1448) (1.0649, -0.68221) (1.0649, 0.68221)};  
\begin{scope}
\draw[black] plot[mark=*, mark options={fill=gray}, mark size=0.6pt] (1.3,1.2)
node[right]{\color{black}\scriptsize $n=1$};
\draw plot[mark=*, mark options={fill=white},mark size=0.65pt] (1.3,1.05)
node[right]{\scriptsize $n=2$};
\draw[gray] plot[mark=*, mark options={fill=gray},mark size=0.6pt] (1.3,0.9)
node[right]{\scriptsize $n=3$};
\draw plot[mark=*, mark options={fill=black},mark size=0.6pt] (1.3,0.75)
node[right]{\scriptsize $n=4$};
\end{scope}
\foreach \x/\xtext in {1/1, 1.333/\frac 43}
\draw[shift={(\x,0)}] (0pt, 1pt) -- (0pt,-1pt) node[below left] {\tiny $\xtext$};
\end{tikzpicture}
\caption{\label{fig4} Zeros of $Q_{10n;1}(z;4/3)$ for $n=1,2,3,4$ in the
window $[-1.5,1.5]\times[-1.5,1.5]$, with $m=2/3$.}
\end{center}
\end{minipage}
\end{figure}

\section*{Acknowledgements}
 The work of the author R.S.C.-S.~was partially supported by Direcci\'on
General de Investigaci\'on Cient\'ifica y T\'ecnica, Ministerio de Econom\'ia
y Competitividad of Spain, under grant MTM2015-65888-C4-2-P.

\bibliographystyle{plain}


\end{document}